\documentclass[12pt, dvipdfmx]{article}
\usepackage[mathscr]{eucal}
\usepackage{amssymb}
\usepackage{latexsym}
\usepackage{amsthm}
\usepackage{amsmath}
\usepackage[dvips]{graphicx}
\usepackage{psfrag}
\usepackage{a4wide}
\usepackage{tikz-cd}
\usepackage{amscd}

\newtheorem{theorem}{Theorem}[section]
\newtheorem{definition}[theorem]{Definition}
\newtheorem{lemma}[theorem]{Lemma}

\newtheorem{corollary}[theorem]{Corollary}
\newtheorem{proposition}[theorem]{Proposition}

\newtheorem{acknowledgment}{Acknowledgment}

\numberwithin{equation}{section}
\newcommand{\C}{{\mathbb C}}
\newcommand{\D}{{\mathbb D}}
\newcommand{\N}{{\mathbb N}}
\newcommand{\cM}{{\mathcal M}}
\newcommand{\cD}{{\mathcal D}}
\newcommand{\R}{{\mathbb R}}
\newcommand{\cH}{{\mathcal H}}
\newcommand{\cF}{{\mathcal F}}

\newcommand{\la}{\langle}
\newcommand{\ra}{\rangle}
\newcommand{\lam}{\lambda}

\begin{document}

\title{
An unbounded approach to the theory of kernel functions
}
\author{
{\sc Michio SETO}\\
[1ex]
{\small National Defense Academy,  
Yokosuka 239-8686, Japan} \\
{\small 
{\it E-mail address}: {\tt mseto@nda.ac.jp}}
}
\date{}

\maketitle
\begin{abstract}
In this paper, we give a new approach to the theory of kernel functions. 
Our method is based on the structure of Fock spaces. 
As its applications, various examples of strictly positive kernel functions are given. Moreover,  
we give a new proof of the universal approximation theorem for the Gaussian kernel. 
\end{abstract}

\begin{center}
keywords: reproducing kernel Hilbert space, strictly positive kernel, Gaussian kernel
\end{center}

\section{Introduction}

Let $X$ be a set, and
let $k$ be a complex-valued function on $X\times X$.  
Then, $k$ is called a kernel function if $k$ is self-adjoint, that is, $k(y,x)=\overline{k(x,y)}$, and  
\[
\sum_{i,j=1}^nc_i\overline{c_j}k(x_i,x_j)\geq 0
\]
for any $n\geq 1$, any $x_1,\ldots, x_n\in X$ and any $(c_1,\ldots, c_n)\in \C^n$. 
In this paper, kernel functions will be called kernels for short. 
Moreover, a kernel $k$ is said to be strictly positive if 
\[
\sum_{i,j=1}^nc_i\overline{c_j}k(x_i,x_j)>0
\]
for any $n\geq 1$, any $n$ distinct points $x_1,\ldots, x_n\in X$ and any $(c_1,\ldots, c_n)\in \C^n\setminus \{\mathbf{0}\}$. 
Strictly positive kernels have received attention not only in functional analysis (see Subsection 3.2 in Paulsen-Raghupathi~\cite{PR})
but also in numerical analysis (see Micchelli~\cite{M}) 
and machine learning (see Guella~\cite{Guella}). 

Now, the purpose of this paper is to introduce a new method for dealing with strictly positive kernels. 
Our method is based on the structure of Fock spaces induced by kernels. 
We will apply the tensor algebra structure of the full Fock space  
to the theory of kernels. This is the main idea of this paper.  
As consequences, we obtain not only various examples of strictly positive kernels (Theorem \ref{thm:3-2} and examples in Section 4),   
but also a new proof of the universal approximation theorem for the Gaussian kernel
 (Theorem \ref{thm:5-1} and Corollary \ref{cor:5-1}).  

This paper is organized as follows. 
In Section 2, 
we introduce huge reproducing kernel Hilbert spaces from power series whose coefficients are strictly positive. 
Although all materials of Section 2 are well known to specialists in Hilbert space operator theory, 
we will give the details for the sake of general readers.   
In Section 3, we investigate relations between 
the strict positivity of kernels and reproducing kernel Hilbert spaces constructed in Section 2. 
In Section 4, various examples of strictly positive kernels are given as applications of results obtained in Section 3. 
In Section 5, we prove the universal approximation theorem for the Gaussian kernel with our method. 

\section{Preliminaries}
Let $k$ be a kernel, and 
let $\cH_k$ be the reproducing kernel Hilbert space generated by $k$. 
We will use the notation $k_x(y)=k(y,x)$ throughout the whole of this paper. 
We fix a sequence $\{a_n\}_{n\geq 0}$ such that $a_n> 0$ for any $n\geq 0$ and 
\[
\sum_{n=0}^{\infty}a_n\|k_x\|_{\cH_k}^{2n}=\sum_{n=0}^{\infty}a_nk(x,x)^n<\infty
\]
for any $x$ in $X$. 
Moreover, we set $\varphi(z)=\sum_{n=0}^{\infty}a_nz^n$.

In this section, we construct a huge reproducing kernel Hilbert space from $\cH_k$ and $\varphi$. 
The contents of this section are well known to specialists. 
For example, see Exercise (k) in p.\ 320 of Nikolski~\cite{Nik}, 
Chapters 5 and 7 in Paulsen-Raghupathi~\cite{PR} and Theorem 8.2 in Saitoh-Sawano~\cite{SS}. 
However, we give the details for the sake of general readers.

Let $\cH_k^n$ denote the reproducing kernel Hilbert space obtained 
by the pull-back construction with the $n$-fold tensor product Hilbert space 
\[
\cH_k^{\otimes n}=\cH_k\otimes \cdots \otimes \cH_k
\]
and the $n$-dimensional diagonal map 
\[
\Delta_n:X\to X^n,\ x\to (x,\ldots,x).
\]
More precisely, $\cH_k^n$ is equal to $\{F\circ \Delta_n:F\in \cH_k^{\otimes n}\}$ as vector spaces and 
its inner product is defined by 
\[
\la f,g \ra_{\cH_k^n}=\la P_{(\ker \Delta_n)^{\perp}}F, 
P_{(\ker \Delta_n)^{\perp}}G\ra_{\cH^{\otimes n}}\quad (f=F\circ \Delta_n,\ g=G\circ\Delta_n),
\]
where we define $\ker \Delta_n=\{F\in \cD^{\otimes n}:F\circ \Delta_n=0\}$ and 
$P_{(\ker \Delta_n)^{\perp}}$ denotes the orthogonal projection onto the orthogonal complement of $\ker \Delta_n$. 
Thus, we identify $\Delta_n$ with the linear map $F\mapsto F\circ \Delta_n$. 
Then, it is easy to see that $k_x^{\otimes n}$ belongs to $(\ker \Delta_n)^{\perp}$. 
Hence, for any function $f=F\circ \Delta_n$ in $\cH_k^n$, we have
\begin{align*}
\la f,k_x^n \ra_{\cH_k^n}
&=\la P_{(\ker \Delta_n)^{\perp}}F, 
P_{(\ker \Delta_n)^{\perp}}k_x^{\otimes n}\ra_{\cH^{\otimes n}}\\
&=\la F, k_x^{\otimes}\ra_{\cH^{\otimes n}}\\
&=F(x,\ldots,x)\\
&=f(x).
\end{align*}  
This concludes that $k_x^n$ is the reproducing kernel of $\cH_k^n$
(for the further details of this construction, see Theorems 5.7 and 5.16 in \cite{PR}). 

Next, 
let $\cF$ denote the Hilbert space with the inner product
\[
\la (f_0,f_1,\ldots)^{\top},(g_0,g_1,\ldots)^{\top} \ra_{\cF}
=\sum_{n=0}^{\infty}a_n\la f_n,g_n \ra_{\cH_k^n},
\]
where 
$f_n$ and $g_n$ are functions in $\cH_k^n$, and 
we set $\cH_k^0=\C$. 
We consider the map $\Gamma$ defined as follows: 
\[
\Gamma: 
\begin{pmatrix}
f_0\\
f_1\\
\vdots
\end{pmatrix}
\mapsto \sum_{n=0}^{\infty}a_nf_n\quad 
\left(
\begin{pmatrix}
f_0\\
f_1\\
\vdots
\end{pmatrix}
\in\cF
\right).
\] 
\begin{proposition}\label{prop:3-1}
$\Gamma$ is a linear map from $\cF$ 
to the vector space consisting of functions on $X$.
Moreover, $\ker \Gamma$ is closed. 
\end{proposition}

\begin{proof}
For any $F=(f_0,f_1,\ldots)^{\top}$ in $\cF$, 
we have 
\begin{align}\label{eq:2-1}
\left| \sum_{\ell=n+1}^ma_{\ell}f_{\ell}(x) \right|
&\leq 
\sum_{\ell=n+1}^m\left|a_{\ell}f_{\ell}(x)\right|\notag \\
&\leq \sum_{\ell=n+1}^ma_{\ell}\|f_{\ell}\|_{\cH_k^{\ell}}\|k_x^{\ell}\|_{\cH_k^{\ell}}\notag \\
&\leq \left(\sum_{\ell =n+1}^ma_{\ell}\|f_{\ell}\|_{\cH_k^{\ell}}^2\right)^{1/2}
\left(\sum_{\ell=n+1}^ma_{\ell}\|k_x^{\ell}\|_{\cH_k^{\ell}}^2\right)^{1/2}\notag \\
&= \left(\sum_{\ell=n+1}^ma_{\ell}\|f_{\ell}\|_{\cH_k^{\ell}}^2\right)^{1/2}
\left(\sum_{\ell=n+1}^ma_{\ell}\|k_x\|_{\cH_k}^{2\ell}\right)^{1/2}.
\end{align}
Hence, 
$\sum_{n=0}^{\infty}a_nf_n(x)$ converges. This concludes that  
$\Gamma$ is a linear map from $\cF$ 
to the vector space consisting of functions on $X$. 
Moreover, by (\ref{eq:2-1}), we have
\begin{equation}\label{eq:2-2}
|(\Gamma F)(x)|\leq \|F\|_{\cF}\left(\sum_{n=0}^{\infty}a_nk(x,x)^n\right)^{1/2}.
\end{equation}
This inequality concludes that $\ker \Gamma$ is closed. 
\end{proof}

By Proposition \ref{prop:3-1}, the pull-back construction can be applied to $\Gamma$. 
\begin{definition}
We define $\varphi(\cH_k)$ as the reproducing kernel Hilbert space obtained by the pull-back construction 
with the linear map $\Gamma$, that is, 
$\varphi(\cH_k)$ is equal to the range of $\Gamma$ as vector spaces and 
its inner product is defined by 
\[
\la f,g \ra_{\varphi(\cH_k)}=\la P_{(\ker \Gamma)^{\perp}}F, 
P_{(\ker \Gamma)^{\perp}}G\ra_{\cF}\quad (f=\Gamma F,\ g=\Gamma G,\ F,G\in \cF). 
\]
\end{definition}

We summarize basic properties of $\varphi(\cH_k)$.  
\begin{proposition}
$\varphi(\cH_k)$ is a reproducing kernel Hilbert space consisting of functions on $X$. 
More precisely, for any $f$ in $\varphi(\cH_k)$, 
there exists a vector $(f_0,f_1,\ldots,)^{\top}$ in $\cF$ 
such that 
\[
f=\sum_{n=0}^{\infty}a_nf_n. 
\]
Moreover, 
the reproducing kernel of $\varphi(\cH_k)$ is 
\[
\sum_{n=0}^{\infty}a_nk_x^n= \varphi(k_x),
\]
that is, 
\[
f(x)=\la f, \varphi(k_x)\ra_{\varphi(\cH_k)}
\]
for any $f$ in $\varphi(\cH_k)$ and any $x$ in $X$.
\end{proposition}

\begin{proof}
We shall show that $\varphi(k_x)$ is the reproducing kernel of $\varphi(\cH_k)$. 
If $(h_0,h_1,h_2,\ldots)^{\top}$ is a vector in $\ker \Gamma$, then
\[
\la (h_0,h_1,h_2,\ldots)^{\top},(1,k_x,k_x^2,\ldots)^{\top} \ra_{\cF}
=\sum_{n=0}^{\infty}a_n\la h_n,k_x^n \ra_{\cH_k^n}=\sum_{n=0}^{\infty}a_nh_n(x)=0.
\]
Hence, $(1,k_x,k_x^2,\ldots)^{\top}$ belongs to $(\ker \Gamma)^{\perp}$. 
Next, for any function $f=\Gamma (f_0,f_1,f_2,\ldots)^{\top}$ in $\varphi(\cH_k)$, we have
\begin{align*}
\la f, \varphi(k_x) \ra_{\varphi(\cH_k)}
&=\la P_{(\ker \Gamma)^{\perp}} (f_0,f_1,f_2,\ldots)^{\top}, 
P_{(\ker \Gamma)^{\perp}}(1,k_x,k_x^2,\ldots)^{\top}\ra_{\cF}\\
&=\la (f_0,f_1,f_2,\ldots)^{\top}, (1,k_x,k_x^2,\ldots)^{\top}\ra_{\cF}\\
&=\sum_{n=0}^{\infty}a_n\la f_n,k_x^n \ra_{\cH_k^n}\\
&=\sum_{n=0}^{\infty}a_nf_n(x)\\
&=f(x).
\end{align*}  
This concludes the proof. 
\end{proof}

\section{Strictly positive kernels}

In this section, we will investigate relations between the strict positivity of $\varphi(k)$ and the structure of $\varphi(\cH_k)$. 
Some of results obtained in Kuwahara-S~\cite{KS} are generalized into the setting of this paper.  

\begin{lemma}[\cite{KS}]\label{lem:3-1}
Let $\psi$ be a function in $\cH_k$. Then, 
the multiplication operator $M_{\psi}$ with symbol $\psi$ 
is a densely defined closable linear operator in $\varphi(\cH_k)$. 
In particular, 
the adjoint operator $M_{\psi}^{\ast}$ is a densely defined closed linear operator in $\varphi(\cH_k)$, and 
every $\varphi(k_x)$ is an eigenfunction of $M_{\psi}^{\ast}$. 
More precisely, 
\[
M_{\psi}^{\ast}\varphi(k_x)=\overline{\psi(x)}\varphi(k_x).
\]
\end{lemma}

\begin{proof}
We define the bounded linear operator $L_{\psi}$ as follows:  
\[
L_{\psi}:
\cH_k^{\otimes n} \to\cH_k^{\otimes n+1}, 
\quad F\mapsto \psi \otimes F.
\]
Then, the following diagram commutes:
\[
\begin{CD}
\cH_k^{\otimes n} @>{L_{\psi}}>> \cH_k^{\otimes n+1} \\
@V{\Delta_n}VV    @VV{\Delta_{n+1}}V \\
\cH_k^n   @>>{M_{\psi}|_{\cH_k^n}}>  \cH_k^{n+1},
\end{CD}
\]
where $\Delta_n$ is identified with the linear map $F\mapsto F\circ \Delta_n$. 
Hence, 
for any function $f_n$ in $\cH_k^n$, 
$\psi f_n$ belongs to $\cH_k^{n+1}$. 
Let $F=(f_0,f_1,\ldots,f_N,0,\ldots)^{\top}$ be a vector with finite support in $\cF$. 
We set $f=\Gamma F$. 
Then, 
\[\psi f=\psi \sum_{n=0}^Na_nf_n=\sum_{n=0}^Na_n\psi f_n=
\sum_{n=0}^Na_{n+1}\dfrac{a_n}{a_{n+1}}\psi f_n=\sum_{n=1}^{N+1}a_n\dfrac{a_{n-1}}{a_n}\psi f_{n-1},
\]
where we note that $\psi f_{n-1}$ belongs to $\cH_k^n$. 
Hence, setting 
\[
G=\left(0, \dfrac{a_0}{a_1}\psi f_0,\dfrac{a_1}{a_2}\psi f_1,\ldots, \dfrac{a_N}{a_{N+1}}\psi f_N,0,\ldots\right)^{\top},
\] 
$G$ belongs to $\cF$ and $\Gamma G=\psi f$, that is, $\psi f$ belongs to $\varphi(\cH_k)$. 
Therefore, $M_{\psi}$ is a densely defined linear operator in $\varphi(\cH_k)$. 
Moreover, 
it is easy to see that $M_{\psi}$ is closable and $M_{\psi}^{\ast}\varphi(k_x)=\overline{\psi(x)}\varphi(k_x)$. 
\end{proof}

We extract the next definition and theorem from the proof of the main theorem in \cite{KS} 
where exponentials of de Branges-Rovnyak kernels were discussed.

\begin{definition}[\cite{KS}]
Let $\cH_k$ be a reproducing kernel Hilbert space on a set $X$. 
Then, $X$ is said to be finitely separated by $\cH_k$ if, 
for any positive integer $n$ and 
for any $n$ distinct points $x_1,\ldots,x_n$ in $X$, 
there exists a function $\psi$ in $\cH_k$ such that 
$\psi(x_i)\neq \psi(x_j)$ whenever $i\neq j$. 
\end{definition}
 
\begin{theorem}[\cite{KS}]\label{thm:3-1}
Let $\cH_k$ be a reproducing kernel Hilbert space on a set $X$. 
If $X$ is finitely separated by $\cH_k$, then $\varphi(k)$ is strictly positive definite. 
\end{theorem}

\begin{proof}
It suffices to show that 
$\{\varphi(k_x)\}_{j=1}^n$ is linearly independent 
for any $n$ in $\mathbb N$ and any $n$ distinct points $x_1,\ldots, x_n$ in $X$. 
Suppose that 
\[
\sum_{j=1}^nc_j\varphi(k_{x_j})=0
\]
for some $n$ in $\mathbb N$, some $n$ distinct points $x_1,\ldots, x_n$ in $X$, 
and some $c_1,\ldots,c_n$ in $\C$. 
Then, for any function $\psi$ in $\cH_k$, 
by Lemma \ref{lem:3-1}, we have 
\begin{align*}
\begin{pmatrix}
1 & \cdots & 1\\
\overline{\psi(x_1)} & \cdots & \overline{\psi(x_n)}\\
\vdots & \vdots & \vdots \\
\overline{\psi(x_1)}^{n-1} & \cdots & \overline{\psi(x_n)}^{n-1}
\end{pmatrix}
\begin{pmatrix}
c_1\varphi(k_{x_1})\\
c_2\varphi(k_{x_2})\\
\vdots\\
c_n\varphi(k_{x_n})
\end{pmatrix}
&=
\begin{pmatrix}
\sum_{j=1}^nc_j\varphi(k_{x_j})\\
\sum_{j=1}^n\overline{\psi(x_j)}c_j\varphi(k_{x_j})\\
\vdots\\
\sum_{j=1}^nc_j\overline{\psi(x_j)}^{n-1}\varphi(k_{x_j})
\end{pmatrix}\\
&=
\begin{pmatrix}
\sum_{j=1}^nc_j\varphi(k_{x_j})\\
M_{\psi}^{\ast}\sum_{j=1}^nc_j\varphi(k_{x_j})\\
\vdots\\
(M_{\psi}^{\ast})^{n-1}\sum_{j=1}^nc_j\varphi(k_{x_j})
\end{pmatrix}\\
&=\mathbf{0}.
\end{align*}
However, by the assumption, there exists a function $\psi$ in $\cH_k$ such that 
\[
\prod_{1\leq i<j\leq n}(\psi(x_i)-\psi(x_j))\neq 0.
\]
Then, the Vandermonde matrix
\[
\begin{pmatrix}
1 & \cdots & 1\\
\overline{\psi(x_1)} & \cdots & \overline{\psi(x_n)}\\
\vdots & \vdots & \vdots \\
\overline{\psi(x_1)}^{n-1} & \cdots & \overline{\psi(x_n)}^{n-1}
\end{pmatrix}
\]
is nonsingular. 
Therefore, we have that 
\[
\begin{pmatrix}
c_1\varphi(k_{x_1})\\
c_2\varphi(k_{x_2})\\
\vdots\\
c_n\varphi(k_{x_n})
\end{pmatrix}=\mathbf{0}.
\]
This concludes that $c_1=\cdots=c_n=0$. 
Indeed, 
if $\varphi(k_{x_j})=0$ as a function for some $1\leq j \leq n$, 
then we would have $f(x_j)=0$ for any $f$ in $\varphi(\cH_k)$.  
However, $\varphi(\cH_k)$ includes $\C$ by definition. 
Hence, $\varphi(k_{x_j})\neq 0$ for any $1\leq j\leq n$. 
\end{proof}

Further, in the case where $\varphi(z)=e^z$, we have the following useful results 
(cf. Theorem 4.1 in Guella~\cite{Guella}). 

\begin{lemma}\label{lem:3-2}
Let $\cH_k$ be a reproducing kernel Hilbert space on a set $X$. 
If $\exp tk$ is strictly positive definite for some $t>0$, 
then so is $\exp (-t\|k_x-k_y\|_{\cH_k}^2)$. 
\end{lemma}

\begin{proof}
Suppose that $\exp tk$ is strictly positive definite. 
Then, since
\begin{align*}
\exp (2t\operatorname{Re}\la k_x,k_y\ra _{\cH_k})
&=\exp (t\la k_x,k_y\ra _{\cH_k})\exp (\overline{t\la k_x,k_y\ra _{\cH_k}})\\
&=\overline{\exp{tk(x,y)}}\exp tk(x,y), 
\end{align*}
$\exp (2t\operatorname{Re}\la k_x,k_y\ra _{\cH_k})$ is strictly positive definite 
by the Schur product theorem (Theorem 7.5.3 in Horn-Johnson~\cite{HJ}). 
Hence, for any distinct points $x_1,\ldots,x_n$ in $X$ and any $(c_1,\ldots, c_n)$ in $\C^n\setminus\{\mathbf{0}\}$, 
we have 
\begin{align*}
\sum_{i,j=1}^nc_i\overline{c_j}\exp (-t\|k_{x_i}-k_{x_j}\|_{\cH_k}^2)
&=\sum_{i,j=1}^nc_i\overline{c_j}e^{-t\|k_{x_i}\|_{\cH_k}^2}e^{-t\|k_{x_j}\|_{\cH_k}^2}
\exp (2t\operatorname{Re}\la k_{x_i},k_{x_j}\ra _{\cH_k})>0. 
\end{align*}
This concludes the proof. 
\end{proof}

\begin{theorem}\label{thm:3-2}
Let $\Phi$ be any map from a set $X$ to a Hilbert space $\cH$. 
If $\Phi$ is injective, then $\exp t\la \Phi(y),\Phi(x) \ra_{\cH}$ is strictly positive definite for any $t>0$. 
Moreover, then, so is $\exp (-t\| \Phi(x)-\Phi(y) \|_{\cH}^2)$. 
\end{theorem}

\begin{proof}
We set $k(x,y)=\la \Phi(y),\Phi(x) \ra_{\cH}$. Then, $k$ is a kernel. 
First, we shall show that $\exp tk$ is strictly positive definite. 
By virtue of Theorem \ref{thm:3-1},
it suffices to show that $X$ is finitely separated by $\cH_{tk}$. 
Let $x_1,\ldots,x_n$ be any distinct points in $X$, 
and let $\cM$ be the subspace generated by $\Phi(x_1),\ldots,\Phi(x_n)$. 
Since $\Phi$ is injective, these are $n$ distinct vectors in $\cH$. 
We shall show that there exists some vector $\mathbf{z}$ in $\cH$ such that 
$\la \Phi(x_i),\mathbf{z} \ra_{\cH}\neq \la \Phi(x_j),\mathbf{z} \ra_{\cH}$ whenever $i\neq j$. 
Without loss of generality, we may assume that $\mathbf{z}$ is a vector in $\cM$. 
For any vector $\mathbf{y}$ in $\cM$,   
if there existed $1\leq i(\mathbf{y})<j(\mathbf{y})\leq n$ 
such that $\la \Phi(x_{i(\mathbf{y})}), \mathbf{y}\ra_{\cH}=\la \Phi(x_{j(\mathbf{y})}), \mathbf{y}\ra_{\cH}$, 
then we would have
\[
\cM \subset \bigcup_{1\leq i< j\leq n} \{\Phi(x_i)-\Phi(x_j)\}^{\perp},
\]
where the above orthogonal complements are taken in $\cM$. However, 
comparing the volumes of $\cM$ and $ \{\Phi(x_i)-\Phi(x_j)\}^{\perp}$ in $\cM$, we have a contradiction. 
Hence, there exists a vector $\mathbf{z}$ in $\cM$ such that
\[
\la \Phi(x_i),\mathbf{z} \ra_{\cH}\neq \la \Phi(x_j),\mathbf{z} \ra_{\cH}\quad (i\neq j).
\]
Moreover, 
setting $\psi(x)=\overline{\la \Phi(x), \mathbf{z} \ra_{\cH}}$ and 
$\mathbf{z}=\sum_{j=1}^nc_j\Phi(x_j)$, we have, 
\[
\psi(x)=\la \mathbf{z}, \Phi(x) \ra_{\cH}= \sum_{j=1}^nc_j\la \Phi(x_j),\Phi(x) \ra_{\cH} =\sum_{j=1}^nc_jk(x,x_j)
=\sum_{j=1}^nc_jk_{x_j}(x).
\]
Hence, $\psi$ belongs to $\cH_{tk}$. 
Thus, we have the first half of the theorem. 
Moreover, since
\begin{align*}
\|k_x-k_y\|_{\cH_k}^2
&=k(x,x)-2\operatorname{Re}k(x,y)+k(y,y)\\
&=\la \Phi(x),\Phi(x)\ra_{\cH}-2\operatorname{Re}\la \Phi(x),\Phi(y) \ra_{\cH}+\la \Phi(y),\Phi(y)\ra_{\cH}\\
&=\|\Phi(x)-\Phi(y)\|_{\cH}^2,
\end{align*}
we have the second half by Lemma \ref{lem:3-2}. 
\end{proof}

\section{Examples}

\subsection{Gaussian kernel}

Let $\cH$ be any Hilbert space. 
Then, it is well known that
\[
\exp\left(-t\|x-y\|_{\cH}^2\right)
\] is strictly positive definite on $\cH\times \cH$ for any $t>0$.  
Indeed, this is a direct consequence of Theorem \ref{thm:3-2}.  

\subsection{Drury-Arveson kernel}
Let $\cH$ be a Hilbert space, and let ${\mathbb B}_{\cH}$ be the open unit ball in $\cH$. 
Then, it is well known that
\[
\frac{1}{1-\la x,y \ra_{\cH}}
\] is strictly positive definite on ${\mathbb B}_{\cH}\times {\mathbb B}_{\cH}$ (see Subsection 7.3.1 in \cite{PR}). 
Indeed, consider the case 
where $k(x,y)=\la x,y \ra_{\cH}$ and $\varphi(z)=\sum_{n=0}^{\infty}z^n$ in Theorem \ref{thm:3-1} .

\subsection{Pseudo-hyperbolic distance}

Let $\D$ be the open unit disk in the complex plane $\C$. 
Then, it is well known that $\D$ is a metric space with the distance function
\[
d(\lam,\mu)=\left| \frac{\lam-\mu}{1-\overline{\mu}\lam} \right|\quad (\lam,\mu \in \D),
\]
which is called the pseudo-hyperbolic distance. 
We shall prove that 
\[
K_t(\lam,\mu)=
\exp\left(-t\left| \frac{\lam-\mu}{1-\overline{\mu}\lam} \right|^2\right)
\] is strictly positive definite on $\D^2$ for any $t>0$. 
We set 
\[
\psi(\lam,\mu)=d(\lam,\mu)^2.
\]
Then, trivially, $\psi$ is real-valued and symmetric. 
Let $H^2$ denote the Hardy space over $\D$, 
and let $s_{\lam}$ be the normalized reproducing kernel of $H^2$, that is, we set 
\[
s_{\lam}(z)=\frac{\sqrt{1-|\lam|^2}}{1-\overline{\lam}z}\quad (\lam \in \D).
\]
Then, for any $n$ in $\N$, $\lam_1,\ldots, \lam_n$ in $\D$, and $c_1,\ldots, c_n$ in $\C$ such that $\sum_{j=1}^nc_j=0$, 
we have
\begin{align*}
\sum_{i,j=1}^n\overline{c_i}c_j\psi(\lam_i,\lam_j)
&=\sum_{i,j=1}^n\overline{c_i}c_j(1-|\la s_{\lam_i}, s_{\lam_j}\ra_{H^2}|^2)\\
&=\sum_{i,j=1}^n\overline{c_i}c_j-\sum_{i,j=1}^n\overline{c_i}c_j|\la s_{\lam_i}, s_{\lam_j}\ra_{H^2}|^2\\
&=-\sum_{i,j=1}^n\overline{c_i}c_j|\la s_{\lam_i}, s_{\lam_j}\ra_{H^2}|^2\\
&=-\sum_{i,j=1}^n\overline{c_i}c_j\overline{\la s_{\lam_i}, s_{\lam_j}\ra_{H^2}}\la s_{\lam_i}, s_{\lam_j}\ra_{H^2}\\
&\leq 0
\end{align*}
by the Schur product theorem. 
Hence, $\psi$ is conditionally negative definite. 
In particular, $\exp (-t\psi)$ is a kernel for any $t>0$ by Schoenberg's generator theorem (Theorem 9.7 in \cite{PR}). 
Moreover, it follows from Proposition 9.3 in \cite{PR} that 
\begin{align}
k(\lam,\mu)
&=-\psi(\lam,\mu)+\psi(\lam,0)+\psi(0,\mu)-\psi(0,0)\notag\\
&=-d(\lam,\mu)^2+|\lam|^2+|\mu|^2\label{eq:1-1}
\end{align}
is a kernel on $\D^2$. 
Let $\cH_k$ denote the reproducing kernel Hilbert space generated by this $k$.  
Then, we have 
\begin{align*}
\dfrac{1}{2}\|k_{\lam}-k_{\mu}\|_{\cH_k}^2
&=\frac{1}{2}(k(\lam,\lam)-2k(\lam,\mu)+k(\mu,\mu))\\
&=\frac{1}{2}\{2|\lam|^2-2(-d(\lam,\mu)^2+|\lam|^2+|\mu|^2)+2|\mu|^2\}\\
&=d(\lam,\mu)^2\\
&=\psi(\lam,\mu).
\end{align*} 
Hence, 
by virtue of Theorem \ref{thm:3-2}, 
in order to prove that $K_t$ is strictly positive definite,
it suffices to show that the map $\Phi:\lam \mapsto k_{\lam}$ is injective. 
However, it is trivial, because 
\[
d(\lam,\mu)^2=\dfrac{1}{2}\|k_{\lam}-k_{\mu}\|_{\cH_k}^2.
\]

\subsection{Word metric}
Let $G$ be a free group with the finite generators $a_1,\ldots, a_N$. 
For any element $g$ in $G$, $|g|$ will denote the word length of $g$. 
Then, it is well known that $d(g,h)=|h^{-1}g|$ defines a metric on $G$, which is called the word metric on $G$. 
Furthermore, in \cite{H}, Haagerup showed that 
\[
K_t(g,h)=
\exp(-t|h^{-1}g|)
\]
is positive semi-definite on $G\times G$ for any $t>0$. 
We shall show that $K_t$ is strictly positive definite. 
Some notations in Lemma 1.2 of \cite{H} are needed.  
We set 
\[
\Lambda_j=\{(g,h)\in G\times G:g^{-1}h=a_j\}
\quad \mbox{and}\quad \Lambda=\bigcup_{j=1}^N\Lambda_j.
\]
Then, $\cH_{\Lambda}$ denotes a Hilbert space 
with an orthonormal basis $\{\mathbf{e}_{(g,h)}\}_{(g,h)\in \Lambda}$ indexed by $\Lambda$.  
Moreover, we set ${\mathbf e}_{(g,h)}=-{\mathbf e}_{(h,g)}$ if $g^{-1}h=a_j^{-1}$ for some $1\leq j \leq N$. 
Let $g$ be an element in $G$, and let 
$g=g_1\cdots g_n$ be the word for $g$.   
Then, we set $f_0=e$, the unit element of $G$, 
and $f_{\ell -1}^{-1}f_{\ell}=g_{\ell}$ for $1\leq \ell \leq n$.  
Let $\Phi$ be the map from $\Lambda$ to $\cH_{\Lambda}$ defined as follows: 
\[
\Phi(g)={\mathbf e}_{(f_0,f_1)}+{\mathbf e}_{(f_1,f_2)}+\cdots+{\mathbf e}_{(f_{n-1},f_n)}.
\]
Then, it is proved that 
\begin{equation}\label{eq:4-2}
\|\Phi(g)-\Phi(h)\|_{\cH_{\Lambda}}^2=|h^{-1}g|
\end{equation}
(see p.\ 283 of \cite{H}). 
Hence, 
by virtue of Theorem \ref{thm:3-2}, 
in order to show that $K_t$ is strictly positive definite, 
it suffices to show that $\Phi: G\to \cH_{\Lambda}$ is injective. 
However, it is trivial by (\ref{eq:4-2}).

\section{Universality}

In this section, we deal with 
real-valued kernels and real Hilbert spaces.  
Let $X$ be a locally compact Hausdorff space, and let $k=k(x,y)$ be a real-valued continuous kernel on $X\times X$. 
Then, $\exp k$ is continuous, and so are all functions in $\exp \cH_k$. 
Let $K$ be a compact subset in $X$, 
and let $C(K)$ be the Banach algebra consisting of all continuous functions on $K$. 
Then it follows from (\ref{eq:2-2}) that
\[
\|f\|_{\infty,K}:=\sup_{x\in K}|f(x)|\leq \|F\|_{\cF}\sup_{x\in K}\exp \frac{k(x,x)}{2}<\infty. 
\]
for every $f$ in $\exp \cH_k$. 
Let $\cF_0$ denote the subspace of $\cF$ consisting of vectors with finite support. 
Then, by Lemma \ref{lem:3-1}, $\Gamma \cF_0$ is an algebra. 
In fact, the following diagram commutes:
\[
\begin{CD}
\cH_k^{\otimes n} @>{L_{\Psi}}>> \cH_k^{\otimes m+n} \\
@V{\Delta_n}VV    @VV{\Delta_{m+n}}V \\
\cH_k^n   @>>{M_{\psi}|_{\cH_k^n}}>  \cH_k^{m+n}
\end{CD}
\]
where $\Psi$ is a function in $\cH_k^{\otimes m}$, and we set 
$\psi=\Psi \circ \Delta_m$ and  
\[
L_{\Psi}:
\cH_k^{\otimes n} \to\cH_k^{\otimes m+n}, 
\quad F\mapsto \Psi \otimes F.
\]
Hence, the closure of $\{f|_K:f\in \Gamma\cF_0\}$ with respect to the norm $\|\cdot\|_{\infty,K}$ is 
a Banach subalgebra of $C(K)$. 
This observation leads us to a new proof of the universal approximation theorem for the Gaussian kernel.
The following is the main theorem of this section. 
(cf. Theorem 9 in Steinwart~\cite{Steinwart} and Theorem 4.1 in Guella~\cite{Guella}). 

\begin{theorem}\label{thm:5-1}
Let $X$ be a locally compact Hausdorff space, 
and let $k=k(x,y)$ be a real-valued continuous kernel on $X\times X$. 
If $\cH_k$ separates any distinct two points in $X$, 
then, for any compact subset $K$ in $X$, any $\varepsilon>0$ and any function $f$ in $C(K)$, 
there exist $c_1,\ldots, c_N$ in $\R$ and $a_1,\ldots,a_N$ in $X$ such that
\[
\left\| f-\sum_{j=1}^Nc_j\exp k(x,a_j)\right\|_{\infty,K}<\varepsilon.
\]
\end{theorem}

\begin{proof}
Let $A$ denote the Banach algebra obtained 
by taking the closure of $\{f|_K:f\in \Gamma \cF_0\}$ with respect to the norm $\|\cdot\|_{\infty,K}$. 
Then, it follows from the Stone-Weierstrass theorem that $A=C(K)$. 
Hence, for any $\varepsilon>0$ and any function $f$ in $C(K)$, 
there exists a vector $G$ in $\cF_0$ such that $\|f-\Gamma G\|_{\infty,K}<\varepsilon$.  
Moreover, there exists a finite linear combination $h$ of reproducing kernels of $\exp \cH_k$ such that  
$\|\Gamma G-h\|_{\exp{\cH_k}}<\varepsilon$. 
Hence, setting 
\[
M_K=\sup_{x\in K}\exp \frac{k(x,x)}{2},
\]  
$M_K$ is finite and 
we have
\begin{align*}
|f(x)-h(x)|
&\leq |f(x)-(\Gamma G)(x)|+|(\Gamma G)(x)-h(x)|\\
&\leq \|f-\Gamma G\|_{\infty,K}+\|\Gamma G-h\|_{\exp \cH_k}\|\exp k_x\|_{\exp \cH_k}\\
&<(1+M_K)\varepsilon
\end{align*}
for any $x$ in $K$. 
Therefore, we have the conclusion. 
\end{proof}

\begin{corollary}\label{cor:5-1}
For any compact subset $K$ in $\R^n$, any $\varepsilon>0$ and any function $f$ in $C(K)$, 
there exist $c_1,\ldots, c_N$ in $\R$ and $a_1,\ldots,a_N$ in $\R^n$ such that 
\[
\left\|f-\sum_{j=1}^Nc_je^{-\|x-a_j\|_{\R^n}^2}\right\|_{\infty,K}<\varepsilon.
\]
\end{corollary}

\begin{proof}
We set $k(x,y)=2\la x,y \ra_{\R^n}$. Then, $k$ is a real-valued continuous kernel on $\R^2$.  
Let $f$ be any function in $C(K)$. 
Then, by Theorem \ref{thm:5-1}, 
there exist $d_1,\ldots, d_N$ in $\R$ and $a_1,\ldots,a_N$ in $\R^n$ such that
\[
\left\| e^{\|x\|_{\R^n}^2}f-\sum_{j=1}^Nd_je^{2\la x,a_j \ra_{\R^n}}\right\|_{\infty,K}<\varepsilon.
\]
We set $c_j=d_je^{\|a_j\|_{\R^n}^2}$. 
Then, we have
\begin{align*}
\left|f(x)-\sum_{j=1}^Nc_je^{-\|x-a_j\|_{\R^n}^2}\right|
&=\left| f(x)- e^{-\|x\|_{\R^n}^2}\sum_{j=1}^Nc_je^{-\|a_j\|_{\R^n}^2}e^{2\la x,a_j \ra_{\R^n}}\right|\\
&=e^{-\|x\|_{\R^n}^2}\left| e^{\|x\|_{\R^n}^2}f(x)- \sum_{j=1}^Nd_je^{2\la x,a_j \ra_{\R^n}}\right|\\
&\leq \left\|  e^{\|x\|_{\R^n}^2}f-\sum_{j=1}^Nd_je^{2\la x,a_j \ra_{\R^n}}\right\|_{\infty,K}\\
&<\varepsilon.
\end{align*}
This concludes the proof. 
\end{proof}

\begin{acknowledgment}\rm 
The author would like to thank J. C. Guella for his valuable comments on the previous version of this paper. 
This paper grew out of my three days lectures in Nagoya University. 
Special thanks are due to Professor Yoshimichi Ueda and Kenta Kojin.  
This work was supported by JSPS KAKENHI Grant Numbers JP20K03646 and JP21K03285. 
\end{acknowledgment}

\end{document}